\documentclass{svmult}

\usepackage{amsmath,amssymb}
\usepackage{bm}
\usepackage{etex}
\usepackage{graphicx}
\usepackage{hyperref}

\usepackage{cleveref}

\DeclareMathOperator{\Cov}{Cov}
\DeclareMathOperator{\Expectation}{\mathbb E}
\DeclareMathOperator{\Maxexp}{\mathcal E}
\DeclareMathOperator{\Prob}{\mathbb P}
\newcommand{\KH}[2]{\operatorname{DH}\left(#1 \middle| #2 \right)}
\newcommand{\Lexp}[1]{\orliczof {(\cosh-1)} #1}

\newcommand{\avalof}[1]{\left\vert#1\right\vert}

\newcommand{\covat}[3]{\Cov_{#1}\left(#2,#3\right)}
\newcommand{\cpoly}[2]{C^{#1}_{\text{poly}}\left(\reals^{#2}\right)}
\newcommand{\derivby}[1]{\frac{d}{d#1}}
\newcommand{\euler}{\mathrm{e}}
\newcommand{\expectat}[2]{\Expectation_{#1}\left[#2\right]}
\newcommand{\expfiberat}[2]{S_{#1}\maxexpat{#2}}
\newcommand{\expof}[1]{\exp\left(#1\right)}
\newcommand{\gaussdensity}{\gamma}
\newcommand{\gaussint}[2]{\int{#1}\ \gaussdensity(#2) \ d#2 \ }
\newcommand{\maxexpat}[1]{\Maxexp\left(#1\right)}
\newcommand{\normat}[2]{\left\Vert#2\right\Vert_{#1}}
\newcommand{\orliczof}[2]{L_{#1}\left(#2\right)}
\newcommand{\orliczpof}[3]{L_{#1}^{#2}\left(#3\right)}
\newcommand{\pderivby}[1]{\frac {\partial} {\partial #1}}
\newcommand{\probat}[2]{\Prob_{#1}\left(#2\right)}
\newcommand{\reals}{\mathbb{R}}
\newcommand{\scalarat}[3]{\left\langle#2,#3\right\rangle_{#1}}

\newcommand{\setof}[2]{\left\{#1 \, \middle| \, #2 \right\}}

\begin{document}

\title*{Information Geometry of smooth densities on the Gaussian space: Poincar\'e inequalities\thanks{The author is supported by de Castro Statistics, Collegio Carlo Alberto, Turin, Italy. He is a member of GNAMPA-INDAM.}}
\author{Giovanni Pistone} 
\institute{de Castro Statistics, Collegio Carlo Alberto, Piazza
  Arbarello 8, 10122 Torino, Italy}
\authorrunning{G. Pistone}
\titlerunning{IG of smooth densities}
\maketitle

\begin{abstract}
We derive bounds for the Orlicz norm of the deviation of a random variable defined on $\reals^n$ from its Gaussian mean value. The random variables are assumed to be smooth, and the bound itself depends on the Orlicz norm of the gradient. We shortly discuss possible applications to non-parametric Information Geometry.  
\end{abstract}

\section{Introduction}
\label{sec:introduction}

In a series of papers
\cite{pistone:2013Entropy,lods|pistone:2015,pistone:2017-GSI2017,pistone:2018-IGAIA-IV}
we have explored a version of the non-parametric Information Geometry
(IG) for smooth densities on $\reals^n$. Especially, we have
considered the IG associated to Orlicz spaces on the Gaussian
space. The analysis of the Gaussian space is discussed, for example,
in \cite{malliavin:1997,nourdin|peccati:2012}. This set-up provides a
simple way to construct a statistical manifold modelled on Banach
spaces of smooth densities. Other modelling options are in fact
available, for example the global analysis methods of
\cite{kriegl|michor:1997}, but we prefer to work with assumptions that
allow for the use of classical infinite dimensional differential
geometry modelled on B-spaces as in \cite{lang:1995}.

The present note focuses on technical results about useful
differential inequalities and does not consider in detail the
applications. However, we have in mind two main examples of potential applications. The first one is the statistical estimation method based on Hyv\"arinen's divergence,
\begin{equation}
  \label{eq:hyvarinen}
  \KH P Q = \frac12 \int \avalof{\nabla \log P(x) - \nabla \log Q(x)}^2 \ P(x) \ dx \ , 
\end{equation}
where $\avalof{\cdot}$ denotes the Euclidean norm of $\reals^n$, $P, Q$ are positive probability densities of the $n$-dimensional Lebesgue space, see in \cite{hyvarinen:2005,lods|pistone:2015}. The second one is the Otto's inner product \cite{otto:2001,lott:2008calculations}, which is defined by
\begin{equation}
  \label{eq:otto}
  \scalarat P f g = \int \nabla f(x) \cdot \nabla g(x) \ P(x) \ dx \ , 
\end{equation}
where $P$ is a probability density and $f,g$ are smooth random variables such that $\expectat P f = \expectat P g = 0$.

We focus on the exponential representation of positive densities $P =
p \cdot \gaussdensity = \euler^{u - K(u)} \cdot \gaussdensity$, where
$\gaussdensity$ is the standard Gaussian density. The sufficient
statistics $u$ is assumed to belong to an exponential Orlicz space and
$\gaussint {u(x)} x = 0$. The set of all such couples $(p,u)$ is
called statistical bundle. There are other ways to represent positive
densities, namely, those that use deformed exponential functions, $p
\propto \exp_A$ \cite{naudts:2011GTh}. This approach is intended to
avoid the difficulty of the exponential growth and, for this reason,
provides a somehow simpler treatment of smoothness, see
\cite{newton:2019LNCS,newton:2019IG}. We do not further discuss here this interesting formalism.

This paper is organized as follows. In \cref{sec:stat-grad-model}, we provide a recap of basic facts about non-parametric IG and introduce the Gaussian case. The results about Poincar\'e-Wirtinger inequalities are gathered in \cref{sec:gauss-poinc-ineq}. This section contains the main contributions of the paper. A collection of simple examples of possible applications concludes the paper.

\section{Statistical bundle modelled on Orlicz spaces}
\label{sec:stat-grad-model}

First, we review below the theory of Orlicz spaces in order to fix convenient notation. The full theory is offered, for example, in \cite[Ch.~II]{musielak:1983} and \cite[Ch.~VII]{adams|fournier:2003}.

\subsection{Orlicz spaces}
\label{sec:orlicz-spaces}

In this paper, we will need the following special type of Young function. Cf. \cite[\S 7]{musielak:1983} for a more general case.

Assume $\phi \in C[0,+\infty[$ is such that: 1) $\phi(0)=0$; 2) $\phi(u)$ is strictly increasing; 3) $\lim_{u \to +\infty} \phi(u) = +\infty$. The primitive function
\begin{equation*}
  \Phi(x) = \int_0^x \phi(u) \ du \ , \quad x \geq 0 \ ,
\end{equation*}
is strictly convex and will be called a \emph{Young function}. Cf. \cite[\S~8.2]{adams|fournier:2003}, where $\phi$ is assumed to be right-continous and non-decreasing.

The inverse function $\psi = \phi^{-1}$ has the same properties 1) to 3) as $\phi$, so that its primitive
\begin{equation*}
  \Psi(y) = \int_0^y \psi(v) \ dv \ , \quad y \geq 0 \ ,
\end{equation*}
is again a Young function. The couple $(\Phi, \Psi)$, is a couple of \emph{conjugate} Young functions. The relation is symmetric and we write both $\Psi=\Phi_*$ and $\Phi = \Psi_*$. The Young inequality holds true,
\begin{equation*}
  \Phi(x) + \Psi(y) \geq xy \ , \quad x,y \geq 0 \ ,
\end{equation*}
and the Legendre equality holds true ,
\begin{equation*}
  \Phi(x) + \Psi(\phi(x)) = x \phi(x) \ , \quad x \geq 0 \ .
\end{equation*}

Here are the specific cases we are going to use:
\begin{align}
&\Phi(x) = \frac {x^p} p \ , \quad \Psi(y) = \frac {y^q}q \ , \quad p,q > 1 \  , \quad \frac1p+\frac1q = 1 \ ;   \label{eq:cases-1} \\
  &\exp_2(x) = \euler^x - 1 - x \ , \quad (\exp_2)_*(y) = (1+y)\log(1+y) - y \ ; \label{eq:cases-2} \\
  &(\cosh-1)(x) = \cosh x - 1 \ , \quad (\cosh-1)_*(y) = \int_0^y \sinh^{-1}(v) \ dv \ ; \label{eq:cases-3} \\
&\operatorname{gauss}_2(x) = \expof{\frac12x^2}-1 \ . \label{eq:cases-4}
\end{align}

Given a Young function $\Phi$, and a probability measure $\mu$, the
Orlicz space $\orliczof \Phi \mu$ is the Banach space whose closed
unit ball is $\setof{f \in L^0(\mu)}{\int \Phi(\avalof f) \ d\mu \leq
  1}$. This defines the Luxemburg norm,
\begin{equation*}
  \normat {\orliczof \Phi \mu} f \leq \alpha \quad \text{if, and only if,} \quad \int \Phi(\alpha^{-1} \avalof f) \ d\mu   \leq 1 \ .
\end{equation*}

From the Young inequality, it holds
\begin{equation*}
  \int \avalof{uv} \ d\mu \leq \int \Phi(\avalof u) \ d\mu + \int \Phi_*(\avalof v) \ d\mu \ .
\end{equation*}
This provides a separating duality $\scalarat \mu u v = \int uv \ d\mu$ of $\orliczof {\Phi} \mu$ and $\orliczof {\Phi_*} \mu$ such that
\begin{equation*}
  \scalarat \mu u v \leq 2 \normat {\orliczof{\Phi} \mu} u  \normat {\orliczof {\Phi_*} \mu} v \ .
\end{equation*}
From the conjugation between $\Phi$ and $\Psi$, an equivalent norm can
be defined, namely, the Orlicz norm
\begin{equation*}
  \normat {{\orliczof {\Phi} \mu }^*} f = \sup \setof{\scalarat \mu f
    g}{\normat {\orliczof {\Psi} \mu} f \leq 1} \ .
\end{equation*}

Domination relation between Young functions imply continuous injection
properties for the corresponding Orlicz spaces. We say that $\Phi_2$
\emph{eventually dominates} $\Phi_1$, written $\Phi_1 \prec \Phi_2$,
if there is a constant $\kappa$ such that $\Phi_1(x) \leq \Phi_2(\kappa x)$
for all $x$ larger than some $\bar x$. As, in our case, $\mu$ is a
probability measure, the continuous embedding $\orliczof {\Phi_2} \mu
\to \orliczof {\Phi_1} \mu$ holds if, and only if, $\Phi_1 \prec
\Phi_2$. See a proof in \cite[Th. 8.2]{adams|fournier:2003}. If
$\Phi_1 \prec \Phi_2$, then $(\Phi_2)_* \prec (\Phi_1)_*$. With
reference to our examples \eqref{eq:cases-2} and \eqref{eq:cases-3},
we see that $\exp_2$ and $(\cosh-1)$ are equivalent. They both are
eventually dominated by $\operatorname{gauss}_2$ \eqref{eq:cases-4}
and eventually dominate all powers \eqref{eq:cases-1}.   

A special case occurs when there exists a function $C$ such that $\Phi(ax) \leq C(a) \Phi(x)$ for all $a \geq 0$. This is true, for example, for a power function and in the case of the functions $(\exp_2)_*$ and $(\cosh-1)_*$. In such a case, the conjugate space and the dual space are equal and bounded functions are a dense set.

The spaces corresponding to case \eqref{eq:cases-1} are ordinary
Lebesgue spaces. The cases \eqref{eq:cases-2} and \eqref{eq:cases-3}
provide isomorphic B-spaces
$\orliczof {(\cosh-1)} \mu \leftrightarrow \orliczof {\exp_2} \mu$
which are of special interest for us as they provide the model spaces
for our non-parametric version of IG, see
\cref{sec:exponential-bundle} below.

A function $f$ belongs to $\orliczof {\cosh-1} \mu$ if, and only if, it is \emph{sub-exponential}, that is, there exist constants $C_1,C_2 > 0$ such that
\begin{equation*}
  \probat \mu {\avalof{f} \geq t} \leq C_1 \expof{-C_2 t} \ , \quad t \geq 0 \ .
\end{equation*}
Sub-exponential random variable are of special interest in
applications because they admit an explicit exponential bounds in the
Law of Large Numbers. Random variables whose square is sub-exponential
are called sub-gaussian. There is a large literature on this subject, see, for example,  \cite{buldygin|kozachenko:2000,vershynin:2018-HDP,wainwright:2019-HDS,siri|trivellato:robust}.

We will be led to use a further notation. For each Young function
$\Phi$, the function $\overline \Phi(x) =  \Phi(x^2)$ is again a Young
function such that $\normat {\orliczof {\overline \Phi} \mu} f \leq
\lambda$ if, and only if, $\normat {\orliczof \Phi \mu} {\avalof f ^2}
\leq \lambda^2$. We denote the resulting space by $\orliczpof \Phi 2
\mu$. For example, $\operatorname{gauss}_2$ and $\overline{\cosh-1}$
are $\prec$-equivalent , hence the isomorphisn $\orliczof
{\operatorname{gauss}_2} \mu \leftrightarrow \orliczpof {(\cosh-1)} 2 \mu$.

As an application of this notation, consider that for each increasing
convex $\Phi$ it holds
$\Phi(fg) \leq \Phi((f^2+g^2)/2) \leq (\Phi(f^2) + \Phi(g^2))/2$. It
follows that when the $\orliczpof \Phi 2 \mu$-norm of $f$ and of $g$
is bounded by one, the $\orliczof \Phi \mu$-norm of $f$, $g$, and
$fg$, are all bounded by one. The need to control the product of two
random variables in $\orliczof {(\cosh-1)} \mu$ appears, for example,
in the study of the covariant derivatives of the statistical bundle,
see
\cite{gibilisco|pistone:98,lott:2008calculations,pistone:2018Lagrange,chirco|malago|pistone:2020-2009.09431}.

\subsection{Calculus of the Gaussian space}
\label{sec:calc-gauss-space}

From now on, our base probability space is the Gaussian probability space $(\reals^n,\gaussdensity)$, $\gaussdensity(z) = (2\pi)^{n/2} \expof{- \avalof z ^2 / 2}$. We will use a few simple facts about the analysis of the Gaussian space, see \cite[Ch.~V]{malliavin:1995}.

Let us denote by $C^k_\text{poly}(\reals^n)$, $k = 0,1,\dots$, the vector space of functions which are differentiable up to order $k$ and which are bounded, together with all derivatives, by a polynomial. This class of functions is dense in $L^2(\gaussdensity)$. For each couple $f,g \in \cpoly 1 n$, we have
\begin{equation*}
  \gaussint {f(x) \ \partial_i g(x)} x =   \gaussint {\delta_i f(x) \ g(x)} x \ ,
\end{equation*}
where the divergence operator $\delta_i$ is defined by $\delta_i f(x) = x_i f(x) - \partial_i f(x)$. Multidimensional notations will be used, for example,
\begin{equation*}
  \gaussint {\nabla f(x) \cdot \nabla g(x)} x = \gaussint {f(x) \ \delta \cdot \nabla g(x)} x \ , \quad f,g \in \cpoly 2 n \ ,
\end{equation*}
with $\delta \cdot \nabla g(x) = x \cdot \nabla g(x) - \Delta g(x)$.

For example, in this notation, the divergence of \cref{eq:hyvarinen} with $P = p\cdot \gaussdensity$, $Q = q \cdot \gaussdensity$, and $p,q \in \cpoly 2 n$, becomes
\begin{multline*} 
  \frac12 \gaussint  {\nabla \log \frac {p(x)}{q(x)} \cdot \nabla \log \frac {p(x)}{q(x)} \  p(x)} x = \\  \frac12 \gaussint  {\log \frac {p(x)}{q(x)} \ \delta \cdot \left( \nabla \log \frac {p(x)}{q(x)} \  p(x)\right)} x \ .
\end{multline*}

The inner product \cref{eq:otto} becomes, with $P = p \cdot \gaussdensity$ and $f,g,p  \in \cpoly 2 n$,
\begin{equation*}
  \gaussint {\nabla f(x) \cdot \nabla g(x) \ p(x)} x  = \gaussint { f(x) \delta \cdot \nabla (g(x) p(x))} x \ .
\end{equation*}

Hermite polynomials $H_\alpha = \delta^\alpha 1$ provide an orthogonal basis for $L^2(\gaussdensity)$ such that $\partial_i H_\alpha = \alpha_i H_{\alpha - e_i}$, $e_1$ the $i$-th element of the standard basis of $\reals^n$. In turn, this provides a way to prove that there is a closure of both operator $\partial_i$ and $\delta_i$ on a domain which is an Hilbert subspace of $L^2(\gaussdensity)$. Such a space is denoted by $D^2$ in \cite{malliavin:1995}. Moreover, the closure of $\partial_i$ is the infinitesimal generator of the translation operator, \cite{malliavin:1997,bogachev:2010}. The space $D^2$ is a Sobolev Space with Gaussian weight based on the $L^2$ norm, \cite{adams|fournier:2003}. By replacing that norm with a $(\cosh-1)$ Orlicz norm, one derives the applications to IG that are presented in \cite{lods|pistone:2015,pistone:2018-IGAIA-IV}.

\subsection{Exponential statistical bundle}
\label{sec:exponential-bundle}

We refer to \cite{pistone:2013GSI,pistone:2018-IGAIA-IV} for the definition of maximal exponential manifold $\maxexpat \gaussdensity$, and of statistical bundle $S \maxexpat \gaussdensity$. Below we report the results that are necessary in the context of the present paper.

A key result is the proof of the following statement of necessary and sufficient conditions, see \cite{cena|pistone:2007} and \cite[Th.~4.7]{santacroce|siri|trivellato:2016}.

\begin{proposition}\label{prop:portmanteau}
  For all $p, q \in \maxexpat \gaussdensity$ it holds $q = \euler^{u - K_p(u)} \cdot p$, where $u \in \orliczof {(\cosh-1)} \gaussdensity$, $\expectat p u = 0$, and $u$ belongs to the interior of the proper domain of the convex function $K_p$. This property is equivalent to any of the following:
\begin{enumerate}\label{portmanteaux}
\item\label{portmanteaux-1} $p$ and $q$ are connected by an open exponential arc;
\item\label{portmanteaux-2} $\Lexp p = \Lexp q$ and the norms are equivalent;
\item\label{portmanteaux-3}
  $p/q \in \cup_{a > 1} L^a(q)$ and $q/p \in \cup_{a>1} L^a(p)$.
\end{enumerate}
\end{proposition}

\Cref{portmanteaux-2} ensures that all the fibers of the statistical bundle, namely $\expfiberat p \gaussdensity$, $p \in \maxexpat \gaussdensity$,  are isomorphic. \Cref{portmanteaux-3} gives a explicit description of the exponential manifold. For example, let $p$ be a positive probability density with respect to $\gaussdensity$, and take $q=1$ and $a = 2$. Then sufficient conditions for $p \in \maxexpat \gaussdensity$ are
\begin{equation*}
  \gaussint {p(x)^2} x < \infty \quad \text{and} \quad \gaussint {\frac1{p(x)}} x < \infty \ .
\end{equation*}
It is interesting to note that there is, so to say, a bound above and a bound below.

\section{Bounding the Orlicz norm with the Orlicz norm of the gradient}
\label{sec:gauss-poinc-ineq}

We discuss now inequalities related to the classical Gauss-Poincar\'e inequality,
\begin{equation}\label{eq:gauss-poincare}
\gaussint {\left(f(x) - \gaussint {f(y)} y \right)^2} x \le \gaussint {\avalof{\nabla f(x)}^2} x \ ,
\end{equation}
where $f \in C^1_\text{poly}(\reals^n)$. A proof is given, for example, in \cite[\S~1.4]{nourdin|peccati:2012} and will follow as a particular case in an inequality to be proved below.

In terms of norms, the inequality above is equivalent to $\normat {L^2(\gaussdensity)}{f - \overline f} \leq \normat {L^2(\gaussdensity)} {\avalof{\nabla f}}$, where $\overline f = \gaussint {f(y)} y$. One can check whether the constant 1 is optimal, by taking $f(x) = \sum_i x_i$ and observing that the two sides both take the value $\sqrt n$.

This is an example of differential inequality of high interest. For example, if $p \in C^2_{\text{poly}}$ is a probability density with respect to $\gaussdensity$, then the $\chi^2$-divergence of $P = p \cdot \gaussdensity$ from $\gaussdensity$ is bounded as follows.
\begin{multline*}
D_{\chi^2}(P \vert \gaussdensity) = \gaussint {(p(x) - 1)^2} x \leq \\ \gaussint {\avalof{\nabla p(x)}^2} x = \gaussint {\delta \cdot \nabla p(x)\ p(x)} x \ , 
\end{multline*}
where $\delta \cdot \nabla p(x) = x \cdot \nabla p(x) - \Delta p(x)$. As $\gaussint {\delta\cdot\nabla p(x)} x = 0$, the RHS is equal to
\begin{multline*}
\gaussint {\delta \cdot \nabla p(x)(p(x)-1)} x \leq \\ \frac12 \gaussint {(\delta\cdot \nabla p(x))^2} x + \frac12 \gaussint {(p(x) - 1)^2} x \ , 
 \end{multline*}
so that, in conclusion, 
\begin{equation*}
  \gaussint {(p(x) - 1)^2} x \leq \gaussint {(\delta\cdot \nabla p(x))^2} x \ .
\end{equation*}

\subsection{Ornstein-Uhlenbeck semi-group}
\label{sec:ornst-uhlenb-semi}

Generalisation of \cref{eq:gauss-poincare} can be derived from the
Ornstein-Uhlenbeck semi-group which is defined on each
$C^k_\text{poly}(\reals^n)$, $k = 0,1,\dots$, by the Mehler formula
\begin{equation}\label{eq:OUsemigroup}
  P_tf(x) = \gaussint {f(\euler^{-t}x + \sqrt{1-\euler^{-2t}}y)} y, \quad t \ge 0,\quad   f \in C^k_\text{poly}(\reals^n) \ ,
\end{equation}
see \cite[V-1.5]{malliavin:1995} and \cite[\S~1.3]{nourdin|peccati:2012}. Notice that $P_0f = f$ and $P_\infty f = \overline f$.

If  $X$, $Y$ are independent standard Gaussian random variables in $\reals^n$, then
\begin{equation}\label{eq:unitary}
  X_t = \euler^{-t}X+\sqrt{1-\euler^{-2t}}Y, \quad Y_t=\sqrt{1-\euler^{-2t}}X-\euler^{-t}Y
\end{equation}
are independent standard Gaussian random variables for all $t \ge 0$. It is well known, and easily checked, that the infinitesimal generator of the Ornstein-Uhlembeck semi-group is $-\delta \cdot \nabla$, that is, for each $f \in \cpoly 2 n$, it holds
\begin{align}
  \derivby t P_tf(x) &= \gaussint {\nabla f(\euler^{-t}x+\sqrt{1 - \euler^{-2t}}y) \cdot \left(- \euler^{-t} x + \frac{\euler^{-2t}}{\sqrt{1-\euler^{-2t}}y}\right)} y \label{eq:generator} \\ &= - (\delta \cdot \nabla) P_t f(x) \label{eq:generator-1} \\ &= - P_t (\delta \cdot \nabla) f(x) \ . \label{eq:generator-2}
\end{align}
See \cite[V.1.5]{malliavin:1995}.

These computations are well known in stochastic calculus, see, for example \cite[\S~5.6]{karatzas|shreve:1991}. In fact, because of \cref{eq:generator-1}, the function $p(x,t) = P_tp(x)$ is a solution of the equation
\begin{equation*}
  \pderivby t p(x,t) - \Delta p(x,t) + x \cdot \nabla p(x,t) = 0 \ , \quad p(x,0)=p(x) \ ,
\end{equation*}
which is the Kolmogorov equation for the diffusion $dX_t = - X_t + \sqrt 2 dW_t$. Similarly, the function $u(x) = \int_0^\infty \euler^{-t} P_tf(x) \ dt$ is a solution of the equation
\begin{equation*}
  \delta \cdot \nabla u(x) + u(x) = f(x) \ .
\end{equation*}

By the change of variable \cref{eq:unitary} and Jensen's inequality, it easily follows that for each convex function $\Phi$ it holds
\begin{equation} \label{eq:OU-non-expansive}
  \gaussint {\Phi(P_tf(x))} x \leq \gaussint {\Phi(f(x))} x \ .
\end{equation}
That is, for all $t \ge 0$, the mapping $f \mapsto P_tf$ is non-expansive  for the norm of each Orlicz space $\orliczof \Phi \gaussdensity$.

We will discuss now a first set of inequalities that involves convexity and differentiation as it is in \cref{eq:gauss-poincare}. This set depends on the following proposition.

\begin{proposition}
For all $\Phi\colon \reals$ convex and all $f \in C^1_{\text{\normalfont poly}}(\reals^n)$, it holds
\begin{multline}\label{eq:first-version}
 \gaussint {\Phi\left(f(x) - \gaussint {f(y)} y\right)} x  \leq \\ 
 \iint \Phi\left(\frac\pi2 \nabla f(x) \cdot y\right) \ \gaussdensity(x) \gaussdensity(y) \ dx dy = \\
\frac1{\sqrt{2\pi}} \iint \Phi\left(\frac\pi2\avalof{\nabla f(x)}z\right) \ \euler^{-z^2/2} \gaussdensity(x) \ dz dx = \gaussint {\widetilde \Phi\left(\avalof{\nabla f(x)}\right)} x \ , 
\end{multline}
where $\widetilde \Phi$ is the convex function defined by
\begin{equation}
  \label{eq:tildePhi}
  \widetilde \Phi(a) = \gaussint {\Phi\left(\frac\pi2 az\right)} z \ .
\end{equation}
\end{proposition}

\begin{proof}
It follows from \cref{eq:OUsemigroup,eq:generator} that
\begin{multline*} 
  f(x) - \gaussint {f(y)} y = P_0f(x) - P_\infty f(x) = - \int_0^\infty \derivby t P_tf(x) \ dt = \\
\frac \pi2 \int_0^\infty p(t) \ dt \int \left(\nabla f(\euler^{-t}x + \sqrt{1-\euler^{-2t}}y)\right) \cdot \left(\sqrt{1-\euler^{-2t}} x - \euler^{-t} y\right) \gaussdensity(y) \ dy \ , 
\end{multline*}
where $p(t) = \frac2\pi\frac{\euler^{-t}}{\sqrt{1-\euler^{-2t}}}$ is a probability density on $t \geq 0$. After that, the application of Jensen inequality and the change of variable \eqref{eq:unitary}, gives \cref{eq:first-version}. See more details in \cite{pistone:2018-IGAIA-IV}.
\end{proof}

The arguments used here differs from those used, for example, in \cite{nourdin|peccati:2012}, which are based on the equation for the infinitesimal generator \cref{eq:generator-1,eq:generator-2}. We will come to that point later. Notice that we can take $\Phi(s) = s^2$ and derive a Poincar\'e inequality with a non-optimal constant $> 1$. 

We can prove now a set of inequalities of the Poincar\'e type. The first example is $\Phi(s) = \euler^{s}$. In such a case, the equation for the moment generating function of the Gaussian distribution gives
\begin{equation*}
  \widetilde \Phi(a) = \gaussint {\expof{\frac\pi2 az}} z = \expof{\frac {\pi^2 a^2}8} \ ,  
\end{equation*}
so that the inequality \eqref{eq:first-version} becomes
\begin{equation*}
 \gaussint {\expof{f(x) - \overline f}} x  \leq \\ \gaussint {\expof{\frac {\pi^2}8 \avalof{\nabla f(x)}^2}} x \ . 
\end{equation*}

More clearly, we can change $f$ to $\frac{2\kappa}\pi f$ and write
\begin{multline}\label{eq:MGF-1}
 \gaussint {\expof{\frac{2\kappa}\pi \left(f(x) - \overline f\right)}} x  \leq \\ \gaussint {\expof{\frac{\kappa^2}2 \avalof{\nabla f(x)}^2}} x = \\ (2\pi)^{n/2} \int \expof{-\frac12\left(\avalof x ^2 - \kappa^2 \avalof{\nabla f(x)}^2\right)} \ dx \ . 
\end{multline}

The inequality above is non-trivial only if the RHS is bounded, that is
\begin{equation*}
  \avalof {\nabla f(x) } < \kappa^{-1} \avalof{x} \ , \quad x \in \reals^n \ ,
\end{equation*}
that is, the function $f$ is Lipschitz. We have found that $f \in C^1(\reals^n)$ and globally Lipschitz implies that $f$ is sub-exponential in the Gaussian space.

The first case of bound for Orlicz norms we consider is the Lebesgue norm, $\Phi(s) = s^{2p}$, $p > 1/2$. In such a case, 
\begin{equation*}
 \widetilde \Phi(a) =  \left(\frac \pi 2\right)^{2p} m(2p)\  a^{2p} \ ,
\end{equation*}
where $m(2p)$ is the $2p$-moment of the standard Gaussian distribution. It follows that
 \begin{equation*}
\normat {L^{2p}(\gaussdensity)} {f - \gaussint {f(y)} y} \leq \frac \pi 2 (m(2p))^{1/2p} \normat {L^{2p}(\gaussdensity)} {\avalof{\nabla f}} \ .
\end{equation*}

The cases $\Phi(a) = a^{2p}$ are special in that we can use the in the proof the multiplicative property $\Phi(ab) = \Phi(a)\Phi(b)$. The argument generalizes to the case where the convex function $\Phi$ is a Young function whose increase is controlled through a function $C$, $\Phi(uv) \leq C(u) \Phi(v)$, and, moreover, such that there exists a $\kappa > 0$ for which
\begin{equation*}
  \gaussint {C\left(\frac \pi 2 \kappa u\right)} u \leq 1 \ , 
\end{equation*}
then \cref{eq:tildePhi} becomes
\begin{equation*}
  \widetilde \Phi(\kappa a) = \gaussint {\Phi\left(\frac \pi2 \kappa a z\right)} z \leq \gaussint {C\left(\frac \pi 2 \kappa z\right)} z \ \Phi(a) \leq \Phi(a)\ .
\end{equation*}

By using this bound in \cref{eq:first-version}, we get
\begin{equation*}
   \gaussint {\Phi\left(\kappa \left(f(x) - \gaussint {f(y)} y\right)\right)} x  \leq \gaussint {\Phi\left(\avalof{\nabla f(x)}\right)} x \ .
 \end{equation*}

 Assume now  that $\normat {\orliczof {\Phi} {\gaussdensity}}
 {\avalof{\nabla f}}\leq 1$ so that the LHS does not exceed 1. Then
 $\kappa \normat {\orliczof {\Phi} {\gaussdensity}} {f -
     \overline f} \leq 1$, which, in turn, implies the inequality
   \begin{equation*}
     \normat {\orliczof {\Phi} {\gaussdensity}} {f-\overline f} \leq \kappa^{-1} \normat {\orliczof {\Phi} {\gaussdensity}} {\avalof{\nabla f}} \ . 
   \end{equation*}

 For example, for $(\exp_2)_*(y) = (1+y)\log(1+  y) - y$ we can take $C(u) = \max(\avalof u,\avalof u ^2)$ and we want a $\kappa > 0$ such that
 \begin{equation*} 
   \gaussint {\max\left(\frac\pi2 \kappa \avalof{u},\left(\frac\pi2 \kappa \avalof{u}\right)^2\right)} u \leq 1 \ .
 \end{equation*}
 Such a $\kappa$ exists because $C$ is $\gaussdensity$-integrable, continous, and $C(0)=0$. For example, as $C(u) \leq u + u^2$, $u \geq 0$, we have
 \begin{multline*}
   \gaussint {C\left(\frac\pi2 \kappa u\right)} u =  2 \int_0^\infty C\left(\frac\pi2 \kappa u\right) \ \gaussdensity(u) \ du \leq \\
\pi \kappa \int_0^\infty u \gaussdensity(u) \ du + \frac{\pi^2}2 \kappa^2 \int_0^\infty u^2 \gaussdensity(u) \ du = \sqrt{\frac\pi2} \kappa + \frac{\pi^2}4 \kappa^2  \end{multline*}
and we can take $k > 0$ satisfying $\sqrt{\frac\pi2} \kappa + \frac{\pi^2}4 \kappa^2 = 1$.

For us, it is of particular interest the case of the Young function $\Phi
= \cosh-1$, for which there is no such bound. Instead, we use \cref{eq:MGF-1} with $\kappa$ and $-\kappa$ to get
\begin{multline}\label{eq:cosh}
 \gaussint {(\cosh-1)\left(\frac{2\kappa}\pi \left(f(x) - \overline f\right)\right)} x  \leq \\ \gaussint {\operatorname{gauss}_2\left(\kappa \avalof{\nabla f(x)}\right)} x \ . 
\end{multline}
Now, if $\kappa = \normat {\orliczof {\operatorname{gauss}_2} {\gaussdensity}}{\avalof{\nabla f}}^{-1}$, then the LHS is smaller or equal then 1, and hence $2\kappa/\pi \normat {\orliczof {\cosh-1}{\gaussdensity}} {f - \overline f} \leq 1$. It follows that
\begin{equation*}
  \normat {\orliczof {\cosh-1}{\gaussdensity}} {f - \overline f} \leq \frac\pi2 \normat {\orliczof {\operatorname{gauss}_2} {\gaussdensity}}{\avalof{\nabla f}} \ .
\end{equation*}

Our last case of this series is the Young function $\operatorname{gauss}_2(x) = \expof{\frac12 \avalof x ^2} - 1$. Assume $f \in \cpoly 0 n \cap \orliczof {\operatorname{gauss}_2} {\gaussdensity}$, that is, there exists a constant $\lambda > 0$ such that
\begin{multline*}
  \gaussint {\operatorname{gauss}_2(\lambda^{-1} f(x))} x = \\
  (2\pi)^{-n/2}  \int \expof{-\frac12\left(\avalof x ^2 - \lambda^{-2} f(x)^2\right)} \ dx - 1 < + \infty \ . 
\end{multline*}

This holds if, and only if, $\avalof x ^2 > \lambda^{-2} \avalof{f(x)}^2$, $x \in \reals^n$, that is, $f$ is bounded by a linear function with coefficient $\lambda > \sup_x \avalof{f(x)}/\avalof{x}$. The case does not seem to be of our interest.

In fact, if we compute $\widetilde{\operatorname{gauss}_2}(\kappa a)$
from \cref{eq:tildePhi}, we find
\begin{multline*}
\gaussint {\operatorname{gauss}_2\left(\frac\pi2 \kappa a z\right)} z
= \\ \frac1{\sqrt{2\pi}} \int \expof{-\frac12\left( 1 -
    \left(\frac\pi2\right)^2 \kappa^{2} a^2\right)z^2} \ dz - 1 = \\
\left(1 - \left(\frac\pi2\right)^2 \kappa^{2} a^2\right)^{-1/2} - 1 \ .
\end{multline*}
if the argument of $(\cdot)^{-1/2}$ is positive, $+\infty$
otherwise. The inequality \cref{eq:first-version} becomes
\begin{equation*}
  \gaussint {\operatorname{gauss}_2\left(\kappa (f(x) - \overline f)\right)} x \leq
  \gaussint {\left(1 - \left(\frac\pi2\right)^2 \kappa^{2}
      \avalof{\nabla f(x)}^2\right)^{-1/2}} x - 1 \ .
\end{equation*}
The function in the RHS does not belong to the class of Young
function we are considering here and would require a special
study.

In the following proposition we give a summary of the inequalities proved so far.

\begin{proposition}
There exists constants $C_1$, $C_2(p)$, $C_3$ such that for all $f \in C^1_\text{\emph{poly}}(\reals^n)$ the following inequalities hold:
 \begin{equation}\label{eq:LlogL-bound}
   \normat {L_{(\exp_2)_*}(\gaussdensity)} {f - \gaussint {f(y)} y} \leq C_1 \normat {L_{(\exp_2)_*}(\gaussdensity)}  {\avalof{\nabla f}} \ .
 \end{equation}
 \begin{equation}\label{eq:2k-bound}
   \normat {L^{2p}(\gaussdensity)} {f - \gaussint {f(y)} y} \leq C_2(p) \normat {L^{2p}(\gaussdensity)} {\avalof{\nabla f}} \ , \quad p > 1/2 \ .
 \end{equation}
 \begin{equation}\label{eq:exp-bound}
           \normat {\Lexp \gaussdensity} {f - \gaussint {f(y)} y} \leq C_3 \normat {\orliczof {\operatorname{gauss}_2} \gaussdensity} {\avalof{\nabla f}} \ .
 \end{equation}
\end{proposition}
Other equivalent norms could be used in the inequalities above. For
example,
$\orliczof {(\exp_2)_*} \gaussdensity \leftrightarrow \orliczof
{(\cosh-1)_*} \gaussdensity$ and
$\orliczof {\operatorname{gauss}_2} \gaussdensity \leftrightarrow
\orliczpof {\cosh-1} 2 \gaussdensity$. We do not care in the present
paper to define explicitly the relevant Gauss-Sobolev spaces as in
\cite{pistone:2018-IGAIA-IV}. But notice the special relevance of the
space based on the norm
$f \mapsto \normat {\orliczpof {\cosh-1} 2 \gaussdensity}
{\avalof{\nabla f}}$.
 
\subsection{Generator of the Ornstein-Uhlenbeck semi-group}
\label{sec:gener-ornst-uhlenb}

We consider now a further set of inequalities which are based on the use of infinitesimal generator $- \delta \cdot \nabla$ of the Ornstein-Uhlenbeck semigroup, see \cref{eq:generator-1,eq:generator-2}. Compare, for example, \cite[\S~1.3.7]{nourdin|peccati:2012}.

We have, for all $f \in \cpoly 2 n$, that
\begin{equation}\label{eq:usegenerator-1}
  f(x) - \overline f = - \int_0^{\infty} \derivby t P_t f(x) \ dt = \int_0^{\infty} \delta \cdot \nabla P_t f(x) \ dt \ .
\end{equation}
Note that
\begin{multline*} 
  \nabla P_t f(x) = \nabla \gaussint {f(\euler^{-t}x + \sqrt{1-\euler^{-2t}}y)} y = \\ \euler^{-t} \gaussint {\nabla f(\euler^{-t}x + \sqrt{1-\euler^{-2t}}y)} y = \euler^{-t} P_t \nabla f(x) \ ,
\end{multline*}
so that
\begin{equation*} 
 P_t \delta \cdot \nabla f(x) = \delta \cdot \nabla P_tf(x) = \euler^{-t} \delta \cdot P_t \nabla f(x) \ .
\end{equation*}
Now, \cref{eq:usegenerator-1} becomes
\begin{equation}
  \label{eq:usegenerator-3}
  f(x) - \overline f = \int_0^{\infty} \euler^{-t} \delta \cdot P_t \nabla f(x) \ dt \ .
\end{equation}

As
\begin{equation*}
  \gaussint {\delta \cdot \nabla f(x)} x = 0 \ ,
\end{equation*}
the covariance of $f,g \in \cpoly 0 n$ is
  \begin{multline*}
    \covat \gaussdensity f g = \\ \gaussint{(f(x) - \overline f)g(x)} x = \gaussint{(f(x) - \overline f)(g(x) - \overline g)} x \ .
  \end{multline*}

It follows that for all $f, g \in \cpoly 2 n$ we derive from \cref{eq:usegenerator-3} 
\begin{equation}\label{eq:OU-selfadjoint}
  \covat \gaussdensity f g = \int_0^{\infty} \euler^{-t} \gaussint {P_t \nabla f(x) \cdot \nabla g(x)} x \ dt \ . 
 \end{equation}

We use here a result of \cite[Prop. 5]{pistone:2018-IGAIA-IV}. Let $\avalof \cdot _1$ and $\avalof \cdot _2$ be two norms on $\reals^n$, such that $\avalof{x \cdot y} \le \avalof x _1 \avalof y _2$. For a Young function $\Phi$, consider the norm of $\orliczof {\Phi} \gaussdensity$ and the conjugate space endowed with the dual norm, 
 \begin{equation*}
   \normat {\orliczof {\Psi,*} \gaussdensity} f = \sup\setof{\int fg\  \gaussdensity}{\int \Phi(g)\ \gaussdensity \leq 1} \ .
 \end{equation*}

The following inequality that includes the standard Poincar\'e case when $\Phi(u) = u^2/2$. 
\begin{proposition}
Given a couple of conjugate Young function $\Phi$, $\Psi$, and norms $\avalof\cdot_1$, $\avalof\cdot _2$ on $\reals^n$ such that $x \cdot y \leq \avalof x _1 \avalof y _2$, $x,y \in \reals^n$, for all $f,g \in C^1_\text{\emph{poly}}(\reals^n)$, it holds
\begin{equation*}
  \avalof{\covat \gaussdensity  f g} \le \normat {\orliczof {\Phi} {\gaussdensity}}{\avalof{\nabla f}_1} \normat {\orliczof {\Psi,*} {\gaussdensity}} {\avalof{\nabla g}_2}  \ .  
\end{equation*}
\end{proposition}
The case of our interest here is $\Phi = \cosh-1$,
$\Psi=(\cosh-1)_*$. As $(\cos-1)_* \prec (\cosh-1)$, it follows, in particular, that
$\covat \gaussdensity f f$ is bouded by a constant times $\normat {\orliczof
  {\cosh-1} \gaussdensity}
{\avalof{\nabla f}}^2$.

\section{Discussion and conclusions}
\label{sec:discussion}

We have collected here a list of possible applications of the information geometry of the Gaussian space that has been introduced in \cite{lods|pistone:2015,pistone:2018-IGAIA-IV} and further developed in the present paper.

\subsection{Sub-exponential random variables}
\label{sec:rand-vari-with}
Let $f \in \cpoly 2 n$ be a random variable of the Gaussian space. Assume moreover that $f$ is globally Lipschitz, that is,
\begin{equation*}
  \avalof{\nabla f(x)} \leq \normat {\operatorname{Lip}(\reals^n)} f\avalof x \,
\end{equation*}
where $\normat {\operatorname{Lip}(\reals^n)} f$ is the Lipschitz semi-norm, that is, the best constant. It follows from \cref{eq:cosh} that $f \in \orliczof {(\cosh-1)} \gaussdensity$ and the norm admits a computable bound.

If $p$ is any probability density of the maximal exponential model of $\gaussdensity$, that is, it is connected to 1 by an open exponential arc, then \cref{prop:portmanteau} implies that $f \in \orliczof {(\cosh-1)} p$, that is, $f$ is sub-exponential under the distribution $P = p \cdot \gaussdensity$. If the sequence $(X_n)_{n=1}^\infty$ is independent and with distribution $p \cdot \gaussdensity$, then the sequence of sample means will converge,
\begin{equation*}
  \lim_{n \to \infty} \frac 1n \sum_{j=1}^n f(X_j) = \gaussint {f(x) \ p(x)} x \ ,
\end{equation*}
with an exponential bound on the tail probability. See, for example, \cite[\S 2.8]{vershynin:2018-HDP}.

\subsection{Hyv\"arinen divergence}
\label{sec:hyvarinen-divergence}
Here we adapt \cite{parry|dawid|lauritzen:2012} to the Gaussian case. Consider the Hyv\"arinen divergence of \cref{eq:hyvarinen} in the Gaussian case, that is, $P = p \cdot \gaussdensity$ and $Q = q \cdot \gaussdensity$. As a function of $q$ is of the form
\begin{multline*}
  H(q) = \frac12 \gaussint {\avalof{\nabla \log p(x)}^2 p(x)} x + \\ \frac12 \gaussint {\avalof{\nabla \log q(x)}^2 p(x)} x - \gaussint {\nabla \log p(x) \cdot \nabla \log q(x) \ p(x) } x \ , 
\end{multline*}
where the first term does not depend on $q$ and the second term is an expectation with respect to $p \cdot \gamma$. As $\nabla \log p = p^{-1} \nabla p$,  the third term equals
\begin{equation*}
  - \gaussint {\delta \cdot \nabla \log q(x) \ p(x)} x \ ,
\end{equation*}
which is again a $p$-expectation. To minimize the Hyv\"arinen divergence we must minimize the $p$-expected value of the local score
\begin{equation*}
  S(q,x) = \frac12 \avalof{\nabla \log q(x)}^2 - \delta \cdot \nabla \log q(x)  
\end{equation*}

If $p$ and $q$ belong to the maximal exponential model of $\gaussdensity$, then $q = \euler^{u - K(u)}$ with $u \in \orliczof {(\cosh-1)} \gaussdensity$ and $\gaussint {u(x)} x = 0$. The local score becomes $\frac12 \avalof{\nabla u}^2 - \delta \cdot \nabla u$. To compute the $p$-expected value of the score with an independent sample of $p \cdot \gaussdensity$ we have interest to assume that the score is in $\orliczof {(\cosh-1)} \gaussdensity$, because this assumption implies the good convergence of the empirical means for all $p$, as it way explained in the section above.

Assume, for example, $\nabla u \in \orliczpof {(\cosh-1)} 2 \gaussdensity$. This implies directly $\avalof{\nabla u}^2 \in \orliczof {(\cosh -1)} \gaussdensity$. Moreover, we need to assume that the $\orliczof {(\cosh -1)} \gaussdensity$-norm of $\delta \cdot \nabla u$ is finite. Under such assumptions it seems reasonable to hope that the minimization on a suitable model of the sample expectation of the Hyv\"arinen score is consistent. 

\subsection{Otto's metric}
\label{sec:ottos-riem-metr}

Let $P = p \cdot \gaussdensity$ with $p$ in the maximal exponential model of $\gaussdensity$. Let $f$ and $g$ be in the $p$-fiber of the statistical manifold, that is, $f,g \in \orliczof {(\cosh -1)} p = \orliczof {(\cosh-1)} \gaussdensity$ and $\gaussint {f(x)} x = \gaussint {g(x)} x = 0$. The Otto's inner product \eqref{eq:otto} becomes
\begin{equation*}
  \gaussint {\nabla f(x) \cdot \nabla g(x) \  p(x)} x = \gaussint {f(x) \ \delta \cdot (p(x) \nabla g(x))} x \ . 
\end{equation*}
The LHS is well defined and regular if we assume $\nabla f, \nabla g \in \orliczpof {(\cosh -1)} 2 \gaussdensity$, because, in such a case, $\avalof{\nabla f}^2, \avalof{\nabla g}^2 \in \orliczof {(\cosh -1)} \gaussdensity = \orliczof {(\cosh -1)} p$. The RHS provides the representation of the inner product in the inner product defined in $\orliczof {(\cosh-1)} \gaussdensity$. Note that the mapping $g \mapsto \delta \cdot (p \nabla g)$ is 1-to-1 if $g$ is restricted by $\gaussint {g(x)p(x)} x = 0$. The inverse of this mapping provides the natural gradient of the Otto's inner product in the sense of \cite{amari:1998natural,li|montufar:2018}.  

\subsection{Conclusion and acknowledgments}

In this paper we have derived bounds of the Orlicz norms of interest
in IG based on the Orlicz norm of the gradient. The schematic examples
above provide, in our opinion, a motivation for further study of this
approach. There is a large literature on Sobolev spaces with weight
that we have, regrettably, not used here. Its study would surely
provide more precise and deep results than those presented here. I
like to thank the Editor and the Referees for the very helpful and
detailed review of this paper.

\end{document}